\documentclass{amsart}
\usepackage{amsmath,amssymb}
\usepackage{graphicx, tikz}
\usepackage{hyperref}
\usepackage{color}
\usepackage[a4paper,  margin=2.5cm]{geometry}

\newtheorem{theorem}{Theorem}[section]
\newtheorem{proposition}[theorem]{Proposition}
\newtheorem{corollary}[theorem]{Corollary}
\newtheorem{conjecture}[theorem]{Conjecture}
\newtheorem{lemma}[theorem]{Lemma}

\theoremstyle{definition}
\newtheorem{definition}[theorem]{Definition}

\newtheorem{example}[theorem]{Example}

\theoremstyle{remark}
\newtheorem{remark}[theorem]{Remark}

\newcommand{\la}{\lambda}

\newcommand{\Mm}{\mathcal{M}}
\newcommand{\Si}{\Sigma}

\def\RR{\mathbb{R}}

\def\ZZ{\mathbb{Z}}

\def\TT{\mathbb{T}}

\newcommand{\cE}{{\mathcal E}}

\newcommand{\cL}{{\mathcal L}}

\newcommand{\cN}{{\mathcal N}}

\newcommand\minus\backslash

\newcommand{\abs}[1]{\left|#1\right|}

\newcommand\lan\langle
\newcommand\ran\rangle

\newcommand{\e}{{\mathrm e}}

\DeclareMathOperator\Div{div}

\DeclareMathOperator\Vol{Vol}

\newcommand{\norm}[1]{\left\lVert#1\right\rVert}

\DeclareMathOperator{\tr}{Tr}

\renewcommand\leq\leqslant
\renewcommand\geq\geqslant
\newlength{\intwidth}

\numberwithin{equation}{section}

\newcommand\BSi\Si 

\newcommand{\dif}{\mathrm{d}}

\newcommand{\Tt}{\mathbb{T}}

\begin{document}

\title[Critical compatible metrics]{On the existence of critical compatible metrics on contact $3$-manifolds}

\author{Y. Mitsumatsu}
\address{Department of Mathematics, Chuo University, 1-13-27 Kasuga Bunkyo-Ku, Tokyo, 112-8551, Japan}
\email{yoshi@math.chuo-u.ac.jp}

\author{D. Peralta-Salas} \address{Instituto de Ciencias Matem\'aticas, Consejo Superior de Investigaciones Cient\'\i ficas, 28049 Madrid, Spain}
\email{dperalta@icmat.es}

\author{R. Slobodeanu}
\address{Faculty of Physics, University of Bucharest, P.O. Box Mg-11, RO--077125 Bucharest-M\u agurele, Romania}
\email{radualexandru.slobodeanu@g.unibuc.ro}

\subjclass[2020]{37D20, 53C15, 53C25, 53D10, 53E50}
%
%
\begin{abstract}
We disprove the generalized Chern-Hamilton conjecture on the existence of critical compatible metrics on contact $3$-manifolds. More precisely, we show that a contact $3$-manifold $(M,\alpha)$ admits a critical compatible metric for the Chern-Hamilton energy functional if and only if it is Sasakian or its associated Reeb flow is $C^\infty$-conjugate to an algebraic Anosov flow modeled
on $\widetilde{SL}(2, \mathbb R)$. In particular, this yields a complete topological classification of compact $3$-manifolds that admit critical compatible metrics. As a corollary we prove that no contact structure on $\TT^3$ admits a critical compatible metric and that critical compatible metrics can only occur when the contact structure is tight.
\end{abstract}

\maketitle

\section{Introduction}

Contact geometry is a vast area of research with a number of important open problems. Different directions of research that have attracted the attention in recent years include the study of global properties of contact structures and its implications for low-dimensional topology, the dynamics of the associated Reeb flows (e.g. the existence of periodic orbits), or applications in mathematical physics (thermodynamics, hydrodynamics, ...). Much less studied are the relations between Riemannian geometry and contact structures, which were first considered by Sasaki and others in the 1960s (see~\cite{blair} and references therein), and then by Chern and Hamilton who proposed an important variational problem in a classical paper~\cite{ch}. They introduced the notion of metrics adapted to contact manifolds. A slightly more general definition of a metric compatible with a contact form was presented by Etnyre, Komendarczyk and Massot in~\cite{ekm}, so let us state it precisely:
\begin{definition}\label{defcomp}
Let $(M,\alpha)$ be a contact $3$-manifold. A Riemannian metric $g$ on $M$ is called \emph{compatible} with $\alpha$ if $\abs{\alpha}_g=1$ and there exists a constant $\theta>0$ such that
$$\ast \dif \alpha  = \theta \alpha\,,$$
where $\ast$ is the Hodge star operator computed with the metric $g$. We say that a contact structure $\zeta$ and a metric $g$ are compatible if there is a defining contact form $\alpha$ for $\zeta$ that is compatible with $g$. It is easy to check that the Riemannian volume defined by $g$ is given by $\mathrm{vol}_g=\frac{1}{\theta}\alpha\wedge \dif\alpha$. All along this article we assume that compatible metrics are $C^\infty$.
\end{definition}

Chern-Hamilton's definition of adapted metric~\cite{ch} corresponds to a compatible metric with $\theta=2$. Furthermore, it coincides with the notion of \emph{associated metric}, customarily used in contact geometry~\cite{blair}. For a good account on the subject, the reader may consult Blair's monograph~\cite{blair}, and the recent articles~\cite{rafal,ekm,ekm2} relating the geometric properties of compatible metrics with purely contact topological information (tightness or overtwistedness).

Given a contact form, it is clear that there is an abundance of compatible metrics, obtained using the polar decomposition of non-singular skew-symmetric matrices, see e.g.~\cite{blairI}. To pinpoint the ``best'' or ``more optimal'' metric, Chern and Hamilton proposed in~\cite{ch} the criterion of minimizing the $L^2$-norm of the torsion $\tau:=\mathcal{L}_R g$, where $R$ is the Reeb field associated to the contact form $\alpha$ (see~\cite[Chapter 10]{blair} for a summary of different proposals of optimal compatible metrics). More precisely, for each constant $\theta>0$, a \emph{critical compatible metric} is defined as a critical point of the \textit{Chern-Hamilton energy functional} $E(g)$ defined on the space $\Mm_\theta(\alpha)$ of $C^\infty$ metrics compatible with $\alpha$ and constant factor $\theta$:
\begin{equation}\label{energy}
E : \Mm_\theta(\alpha) \to [0,\infty), \qquad E(g) = \int_M \abs{\tau}^2_g \mathrm{vol}_g\,.
\end{equation}
A remarkable property is that, if they exist, critical compatible metrics are always (local) minima of $E$, cf.~\cite{deng}, so they can be understood as (locally) optimal metrics for the aforementioned energy functional. As an aside remark, we observe that the specific value of $\theta>0$ is irrelevant for the problem of the existence of critical compatible metrics, see Remark~\ref{R:fixte} below.

Examples of critical compatible metrics include Sasakian metrics (i.e., compatible metrics for which $R$ is a Killing field), for which the energy functional is zero (a global minimum), the standard metric on the tangent sphere bundle of a Riemannian manifold of constant curvature $-1$, cf. \cite{blairIII}, and certain left invariant metrics on compact quotients of $\widetilde{SL}(2, \RR)$, cf. \cite{perr}; interestingly enough in all these cases the energy density $\abs{\tau}^2_g$ is constant. We show in our main theorem that these are essentially all possible critical compatible metrics.

In this article we are concerned with the existence of critical compatible metrics for \emph{general contact manifolds}, which is motivated by an old conjecture of Chern and Hamilton~\cite{ch} stating that, on a closed contact 3-manifold $(M,\alpha)$ whose corresponding Reeb vector field induces a Seifert foliation, there always exists a critical compatible metric. Even taking into account that the Euler-Lagrange equations that they derived turned out to be flawed and the error was lately corrected by Tanno~\cite{tan}, the conjecture maintained its interest and eventually it was solved in the affirmative in~\cite{ruk}. It is interesting to notice that, virtually at the same time and independently, Blair proposed in~\cite{blairI} an equivalent energy functional and in the same article he proved that the only critical compatible metrics are Sasakian if the Reeb flow is regular (i.e., if $M$ is a circle bundle and the Reeb field is tangent to the fibres). Besides the affirmative solution of the Chern-Hamilton conjecture, a number of references have focused on the identification of non-trivial critical metrics on specific contact manifolds and their stability. Recently, the interest in the conjecture has been revived by the works of Perrone and Hozoori~\cite{perr,hozo,perro,hozo1} who also established the connection with Anosovity of the associated Reeb fields. In this direction, Hozoori formulated in~\cite{hozo} a natural extension of the original Chern-Hamilton conjecture letting the contact form to be arbitrary, but associated with a fixed contact distribution:

\begin{conjecture}[Generalized Chern-Hamilton conjecture] \label{conjectgeneral}
For any closed contact 3-manifold $(M, \zeta)$, there exists a compatible metric that realizes the minimum (among compatible metrics) of the Chern-Hamilton energy functional.
\end{conjecture}

In this article we disprove this conjecture by showing that a necessary (and sufficient) condition for the existence of a critical compatible metric is that it is Sasakian or the associated Reeb flow is algebraic Anosov; in particular, this yields a topological classification of $3$-manifolds admitting critical compatible metrics, and hence a wealth of manifolds for which no contact structure admits a critical compatible metric. This is precisely stated in the following main theorem and its remark. In the statement we use a special class of bicontact structures which is defined in Section~\ref{S.proof}.

\begin{theorem}\label{T.main}
A closed contact $3$-manifold $(M,\alpha)$ admits
a critical compatible metric $g$ if and only if:
\begin{itemize}
\item It supports a Sasakian metric, or
\item Its associated Reeb field is an Anosov flow
which is supported by a calibrated bicontact structure.
This is equivalent to the Anosov flow being $C^\infty$-conjugate
to one of the algebraic Anosov flows modeled
on $\widetilde{SL}(2, \mathbb R)$, and $M$ diffeomorphic to a compact quotient of $\widetilde{SL}(2, \mathbb R)$.
\end{itemize}
In the former case $\mathcal L_Rg=0$ and in the latter the norm of $\mathcal L_R g$ is constant on~$M$. Moreover, any critical compatible metric $g$ is a global minimizer of the Chern-Hamilton energy functional.
\end{theorem}

\begin{remark}\label{detailsMain}
The manifold in the above theorem carries one of the 8 geometries
in the sense of Thurston. In the Sasakian case, according to the classification in~\cite{geigeSasa}, the manifold is Seifert fibred and admits an $\mathbb{S}^3$-geometry, a $\mathrm{Nil}^3$-geometry or an $\widetilde{SL}(2, \mathbb R)$-geometry, and the structures are left invariant.  In the Anosov case, the manifold admits an $\widetilde{\mathit{SL}}(2, \mathbb R)$-geometry,
which might not be left-invariant, but up to double covering it is a quotient of $\widetilde{\mathit{SL}}(2, \mathbb R)$ by a right action of a cocompact discrete subgroup of $\widetilde{\mathit{SL}}(2, \mathbb R)$. A 3-manifold with left invariant $\widetilde{\mathit{SL}}(2, \mathbb R)$-geometry
admits two different kinds of critical compatible metrics. Therefore, Theorem~\ref{T.main} yields a complete topological characterization of contact $3$-manifolds admitting critical compatible metrics. In particular, no contact structure on $\TT^3$ admits a critical compatible metric and no overtwisted contact structure admits a critical compatible metric.
\end{remark}




The proof of Theorem~\ref{T.main} is presented in Section~\ref{S.proof}. Before the proof of the main theorem, we prove some instrumental lemmas in Section~\ref{S.prelim}. Finally, corollaries and other applications are presented in Section~\ref{S.appli}.

\section{Preliminary results}\label{S.prelim}
In this section we recall some useful properties of contact $3$-manifolds endowed with a compatible metric, and prove two auxiliary lemmas on critical compatible metrics that are instrumental in the proof of Theorem~\ref{T.main}.

Let $M$ be an oriented and closed smooth $3$-manifold. A contact form on $M$ is a $1$-form $\alpha$ such that $\alpha \wedge \dif \alpha \neq 0$ on $M$. A $2$-plane field $\zeta \subset TM$ is a (coorientable) contact structure on $M$ if there is a contact form $\alpha$ such that $\zeta=\ker \alpha$. Such a $2$-plane field is a maximally non-integrable distribution on $M$, also called a contact distribution. The Reeb field $R$ associated to the contact form $\alpha$ is the unique vector field determined by the conditions
\[
\alpha(R)=1\,,\qquad i_R\dif\alpha=0\,.
\]
Notice that any Reeb field preserves the volume form $\alpha \wedge \dif \alpha$ on $M$.

Given a contact $3$-manifold $(M,\alpha)$, we consider the space of compatible metrics as introduced in Definition~\ref{defcomp}. The following remark shows that we can fix the constant $\theta$ without any loss of generality:

\begin{remark}[Fixing $\theta$]\label{R:fixte}
Let $g$ be a compatible metric with constant $\theta$. Then~\cite[Proposition 2.1]{ekm} the Reeb field $R$ is $g$-orthogonal to $\ker \alpha=\zeta$ and $g(R,R)=1$, so that the metric has the splitting
$$g=g^\zeta + \alpha\otimes\alpha\,,$$
where $g^\zeta$ is a smooth degenerate quadratic form on $TM$ satisfying $g^\zeta(R,\cdot)=0$. It is easy to check that the new metric
$$g' := \frac{\theta}{\theta^\prime}g^\zeta + \alpha\otimes\alpha$$ is still a metric compatible with $\alpha$, but with constant $\theta^\prime$. Moreover we can easily infer that the Chern-Hamilton energy functional is given by
$$\theta'E(g')= \theta E(g)\,.$$
Accordingly, if we normalize the energy~\eqref{energy} by multiplying it with the constant $\theta$ corresponding to $g$, then it does not depend anymore on $\theta$ and we can restrict our variational problem (for which $\theta$ is fixed) to compatible metrics with $\theta =2$, as in the classical definition of Chern and Hamilton~\cite{ch}.
\end{remark}

In view of this remark, since the existence of critical compatible metrics for some constant $\theta_0$ implies the existence for any other constant $\theta$, throughout the rest of the paper we shall work with compatible metrics whose factor is $\theta = 2$.

On a contact manifold $(M, \alpha)$, it is well known~\cite[Theorem 4.4]{blair} that there exists an endomorphism $\phi$ of $TM$ (a $C^\infty$ $(1,1)$-tensor) and a Riemannian metric $g$ such that, for any vector fields $X,Y$ on $M$,
\begin{equation}\label{assocmet}
\phi^2= - I + \alpha \otimes R\,, \quad
g(\phi X, \phi Y)=g(X, Y) - \alpha(X)\alpha(Y)\,, \quad
\tfrac{1}{2}\dif \alpha (X, Y) = g(X, \phi Y)\,.
\end{equation}
One can easily prove that $\abs{\alpha}_g=1$, $g(R, X)=\alpha(X)$ and $\tfrac{1}{2}\alpha \wedge \dif \alpha= \mathrm{vol}_g$, and therefore we have $\ast \dif \alpha = 2 \alpha$, so $g$ is a compatible metric for $\alpha$. In fact, this provides a method to construct a compatible metric by starting with a $(1,1)$-tensor $\phi$ that satisfies $\phi^2= - I + \alpha \otimes R$ (so actually a complex structure on the contact planes, extended along the Reeb field direction by $\phi R =0$) and which is compatible with $d\alpha$, i.e.,
\begin{equation}\label{phi2}
d\alpha(\phi X, \phi Y)=d\alpha(X, Y), \quad d\alpha(\phi X, X) >0\,,
\quad X, Y \in \zeta=\ker \alpha,
\end{equation}
by setting
\begin{equation}\label{defcompatmetric}
g(X,Y):=\tfrac{1}{2}\dif \alpha(\phi X, Y) + \alpha(X) \alpha(Y)\,,
\end{equation}
for any vector fields $X,Y$ on $M$; see also~\cite[Proposition 2.1]{ekm}. Conversely, given any compatible metric $g$ there is a $C^\infty$ associated $\phi$ that satisfies the conditions~\eqref{assocmet}; such a $(1,1)$-tensor $\phi$ is not unique, but its freedom will not be relevant for our purposes.

We observe that, with respect to the compatible metric $g$, $R$ is geodesic (that is, $\nabla_R R =0$) and solenoidal (that is, $\mathcal L_R\mathrm{vol}_g=0$).

Let $g$ be a metric compatible with $(M,\alpha)$ and $\phi$ an associated $(1,1)$-tensor, then we introduce two important objects in the analysis of the existence of critical compatible metrics: the $(1,1)$-tensor
$$h:=\frac{1}{2}\mathcal{L}_R \phi\,,$$
which anti-commutes with $\phi$, that is $h\phi + \phi h=0$ (as we can easily see by applying $\mathcal{L}_R$ to the first equality in \eqref{assocmet}), and the \textit{torsion} tensor
$$\tau:=\mathcal{L}_R g\,,$$
which appears in the definition of the Chern-Hamilton energy functional.

The following lemma is not hard to prove, see~\cite{blair} for details:

\begin{lemma}\label{basicproperties}
The following properties hold:
\begin{itemize}
\item[(i)] $h$ and $\tau$ are related through the relation $\tau(\cdot, \cdot) = 2g(h\phi \cdot, \cdot)$.
\item[(ii)] The  kernel of $h$ is at least 1-dimensional since it contains the Reeb field, i.e., $hR=0$.
\item[(iii)] For any vector field $X$ on $M$, we have the identity
\begin{equation}\label{nablaReeb}
\nabla_{X} R = -\phi X -\phi h X\,,
\end{equation}
\item[(iv)] $h$ is a symmetric operator and $\tr h=0$.
\end{itemize}
\end{lemma}

Since $h\phi+\phi h=0$, it is clear that if $X$ is an eigenvector of $h$ with eigenvalue $\lambda$ then $\phi X$ is an eigenvector with eigenvalue $-\lambda$. In fact, according to~\cite[Lemma 2.3]{pant}, $h$ can be consistently diagonalized at each point of an open and dense subset $\mathcal N$ of $M$. So around each point $p\in\mathcal N$ there exists a local orthonormal frame with respect to which $h$ reads as
$$\left( {\begin{array}{ccc}
   0 & 0 & 0 \\
   0 & -\lambda & 0 \\
   0 & 0 & \lambda
  \end{array} } \right)\,,$$
and its squared norm can be computed as
$$\abs{h}_g^2=2\lambda^2\,.$$
Since $\abs{h}_g^2$ is a global object, we conclude that there exists a globally defined smooth non-negative function $\lambda^2$ such that, around any point where it does not vanish, its square root is an eigenvalue of $h$. An important observation is that, using Lemma~\ref{basicproperties}(i), we easily infer that
\begin{equation}\label{eq.lam}
8\lambda^2 = \abs{\mathcal{L}_R g}_g^2\,,
\end{equation}
so the function $\lambda^2$ does not really depend on the choice of the $(1,1)$-tensor $\phi$ (recall we said it is not unique), but it only depends on $\alpha$ and the compatible metric $g$.

Now we turn our attention to critical compatible metrics. Tanno proved in~\cite{tan} the following characterization of critical compatible metrics:
\begin{proposition}
A compatible metric is a critical point of the Chern-Hamilton energy functional~\eqref{energy} if and only if it satisfies the equation:
\begin{equation}\label{critmetric}
\nabla_R h = 2h\phi\,,
\end{equation}
which is equivalent to
\begin{equation*}
\left(\nabla_R \mathcal{L}_R g\right)(\cdot, \cdot)=2\mathcal{L}_R g (\phi \cdot, \cdot)\,.
\end{equation*}
\end{proposition}
The following lemma is key in the proof of Theorem~\ref{T.main}, because it shows a connection between the function $\lambda^2$ of a critical compatible metric and the dynamics of the Reeb field $R$:

\begin{lemma}\label{1stintegral}
Let $(M,\alpha)$ be a contact $3$-manifold and assume that $g$ is a critical compatible metric. Then the function $\lambda^2 \in C^\infty(M)$ is a first integral of $R$, i.e., $i_Rd(\lambda^2)=0$.
\end{lemma}
\begin{proof}
In the Euler-Lagrange equation~\eqref{critmetric} take the scalar product with the $(1,1)$-tensor $h$ on both sides,
$$\left\langle \nabla_R h , h\right\rangle_g = 2\langle h\phi, h \rangle_g\,,$$
which, with respect to any local orthonormal frame $\{e_i\}_{i=1,2,3}$ in a neighborhood $U$ of a point $p\in M$, reads as
$$
\sum_{i=1}^3 g \left((\nabla_R h)(e_i) , he_i\right) = 2\sum_{i=1}^3 g(h(\phi e_i), he_i)\,.
$$
Then, notice that the right hand side of this equality is zero, which follows from the fact that $h\phi=-\phi h$ and that $X$ is $g$-orthogonal to $\phi X$ for any vector field $X$ on $M$. Moreover, expanding the covariant derivative $\nabla_R$ in the left hand side, we get
\begin{align*}
\sum_{i=1}^3 \Big( g \left(\nabla_R he_i , he_i\right)- g(h\nabla_R e_i, he_i)\Big)&=\tfrac{1}{2}R(\abs{h}_g^2) - \sum_i g(\nabla_R e_i, h^2e_i)\\
&=R(\lambda^2)- \sum_i g(\nabla_R e_i, h^2e_i)\,.
\end{align*}
Finally, assume that $p\in \cN$, the open and dense set introduced before, which allows us to take the basis $\{e_i\}_{i=1,2,3}$ formed by eigenvectors of the tensor $h$. In this case we easily see that
$$
g(\nabla_R e_i, h^2e_i)=\lambda^2 g(\nabla_R e_i, e_i)= \frac{1}{2}\lambda^2 R (g(e_i, e_i))=0\,,
$$
for every $i$, thus implying that $R(\lambda^2)=0$, as we wanted to show.
\end{proof}

\section{Proof of Theorem~\ref{T.main}}\label{S.proof}

We divide it in three steps, in Section~\ref{SS.dir} we prove that the existence of a critical compatible metric implies that the contact manifold is Sasakian or the associated Reeb flow is supported by a smooth calibrated bicontact structure and, moreover, it is $C^\infty$-conjugate to an algebraic Anosov flow on a compact quotient of $\widetilde{SL}(2, \RR)$. The converse implication, i.e., that a contact form whose associated Reeb flow is supported by a calibrated bicontact structure admits a critical compatible metric is established in Section~\ref{SS.conv}, where we also complete the proof of the equivalence between a Reeb flow being supported by a calibrated bicontact structure and being $C^\infty$-conjugate to an algebraic Anosov flow modeled on $\widetilde{SL}(2, \RR)$. Finally, in Section~\ref{SS.min} we show that any critical compatible metric is a global minimizer of the Chern-Hamilton energy functional.

We begin by recalling some definitions and by introducing the notion of a calibrated bicontact structure supporting a Reeb field, closely related to the notion of $(-1)$-Cartan structure~\cite{perro}.  In what follows we employ the notation $\Omega := \frac{1}{2}\alpha \wedge \dif \alpha$.

A flow $\varphi_t$ on a closed 3-manifold $M$ is said to be \textit{Anosov} if there is a Riemannian metric on $M$, a constant $0<\Lambda<1$ and line subbundles $E^s$ and $E^u$ of $T M$ such that $T M=E^s \oplus E^u \oplus \operatorname{span}\{X\}$, where $X=\dot{\varphi}_t$, $d \varphi_t\left(E^{s, u}\right)=E^{s, u}$ for every $t \in \mathbb{R}$ and
$$
\left\|\left.d \varphi_t\right|_{E^s}\right\| \leq \Lambda^t\,, \qquad
\left\|\left.d \varphi_{-t}\right|_{E^u}\right\| \leq \Lambda^t\,, \quad \forall t>0\,,
$$
where the operator norm $\norm{\cdot}$ above is induced by the Riemannian metric on $M$. It is well known that the plane fields $\cE^s:= E^{s}\oplus \operatorname{span}\{X\}$ and $\cE^u:= E^{u}\oplus \operatorname{span}\{X\}$ are integrable, and generally only of class $C^1$.

A \textit{bicontact structure} on a 3-manifold $M$ is defined as a pair of transverse contact plane fields $(\zeta_1, \zeta_2)$ defined by 1-forms $\eta_1$ and $\eta_2$ such that $\eta_1 \wedge d\eta_1$ and $\eta_2 \wedge d\eta_2$  are volume forms on $M$  of opposite orientations. We say that a vector field $X$ is \textit{supported} by such a bi-contact structure if $X\in \ker\eta_1\cap\ker\eta_2$.

\begin{definition}\label{def:cali}
The associated Reeb flow $R$ of a contact manifold $(M,\alpha)$
is supported by a \textit{calibrated bi-contact structure} $(\zeta_1, \zeta_2)$
if there exists a transverse pair of positive and negative contact plane fields $\zeta_1$ and $\zeta_2$ on $M$
which are tangent to the Reeb flow $R$ satisfying the following conditions.
Any point admits a neighborhood on which there exist
$C^\infty$ defining contact forms $\eta_1,\eta_2$ for $\zeta_1$ and $\zeta_2$ satisfying
\begin{equation}\label{cali}
\begin{split}
&\eta_1\wedge d\eta_1=-\eta_2\wedge d\eta_2=\varkappa\Omega\, ,\\
&\eta_1\wedge d\eta_2=\eta_2\wedge d\eta_1=0 \, ,\\
&\alpha\wedge\eta_1\wedge\eta_2=\Omega\, ,
\end{split}
\end{equation}
for some constant $\varkappa> 0$.
\end{definition}

Notice that if the pair of local defining 1-forms $(\eta_1,\eta_2)$
satisfies the above conditions, so does the pair $(-\eta_1,-\eta_2)$.
If $\zeta_1$ and $\zeta_2$ are co-orientable (if one is co-orientable, so is the other), we can take a pair of global defining 1-forms $(\eta_1,\eta_2)$.
If not, instead of considering both pairs $\pm(\eta_1,\eta_2)$,
we can take the double covering $(\widetilde{M},\widetilde{\alpha})$ of $(M,\alpha)$ so that $\zeta_1$ and $\zeta_2$ lift to co-orientable contact bundles and take a global pair. As the involution $\tau$ for the double covering acts as changing the sign of global defining 1-forms, those forms descends to  $\pm(\eta_1,\eta_2)$.
Therefore in any case, this notion is well-defined for  $(\zeta_1, \zeta_2)$ and
the local pairs $\pm(\eta_1,\eta_2)$ are determined.

\begin{remark}
If $R$ is Anosov, then \cite[Proposition 2]{mitsu} it is supported by a bi-contact structure $(\zeta_1, \zeta_2)$, $\zeta_i=\ker \eta_i$, which means that $R\in \ker\eta_1\cap\ker\eta_2$, and $\eta_1\wedge d\eta_1= \varkappa_1 \Omega$, $\eta_2\wedge d\eta_2=-\varkappa_2\Omega$ for some positive functions $\varkappa_{1},\varkappa_2$. Recall that there is a certain amount of flexibility in choosing this bi-contact structure: in~\cite[$\S 5$]{hozo1} it was proved that we can always assume that $\eta_1 \wedge d \eta_1 = -\eta_2 \wedge d \eta_2$  and $\eta_1 \wedge d \eta_2 =\eta_2 \wedge d \eta_1 =0$, paying the price of taking contact forms that are only of class $C^1$.

In the case of a Reeb-Anosov flow, our first condition in Equation~\eqref{cali} is that $\varkappa_1 = \varkappa_2 \equiv$ constant. Imposing the third condition $\alpha\wedge\eta_1\wedge\eta_2=\Omega$ has the following consequence: $\imath_R \Omega = \eta_1\wedge\eta_2$ and by taking the exterior derivative and using Cartan's formula,  we get $\mathcal{L}_R \Omega = d\eta_1\wedge \eta_2 - \eta_1\wedge d\eta_2$. But $\mathcal{L}_R \Omega=0$ so we must have $d\eta_1\wedge \eta_2 = \eta_1\wedge d\eta_2$. Our second condition in~\eqref{cali} asks that they are both identically zero.
\end{remark}

\subsection{Necessary conditions for the existence of critical compatible metrics}\label{SS.dir}

First we prove the following fundamental lemma, which shows that the Reeb field is supported by a special bi-contact structure provided that the contact form admits a critical compatible metric. The proof of this result is modeled on previous works of Perrone~\cite{perr, perro}, but we include it for the sake of completeness.

\begin{lemma}\label{bi-contactsup}
Let $(N,\alpha)$ be a compact contact $3$-manifold,
possibly with boundary.
Assume that $g$ is a critical compatible metric
such that the function $\lambda^2$ introduced in Section~{\rm 2} 
is nowhere vanishing.
Then the associated Reeb field $R$ is supported by a $C^\infty$
bi-contact structure $(\zeta_1, \zeta_2)$
with a defining pair of 1-forms $\pm(\eta_1, \eta_2)$ that satisfies:
\begin{equation}\label{bic0}
\begin{split}
&\eta_1\wedge d\eta_1=-\eta_2\wedge d\eta_2=\lambda\Omega\,,\\
&\eta_1\wedge d\eta_2=\eta_2\wedge d\eta_1=0\,, \\
&\alpha \wedge \eta_1 \wedge \eta_2 = \Omega.
\end{split}
\end{equation}
Moreover we also have:
\begin{equation}\label{bic1}
\mathcal{L}_R \eta_1 = -\lambda \eta_2, \qquad \mathcal{L}_R \eta_2 = -\lambda \eta_1\,,  
\end{equation}
and the critical metric $g$ reads:
\begin{equation}\label{g_eta}
g =\alpha \otimes \alpha +
\eta_1 \otimes \eta_1 + \eta_2 \otimes \eta_2\,.
\end{equation}
\end{lemma}

\begin{proof}
The assumption on $\la^2$ implies that the $(1,1)$-tensor $h$ is nowhere vanishing and has three distinct eigenvalues. Consider first the case that the eigenspace line bundles $\ell_1$ and $\ell_2$
of the eigenvalues $-\lambda$ and $\lambda$ of the $(1,1)$-tensor $h$
are trivial (if one is trivial so is the other). Accordingly, let $\left\{R, e_1, e_2 = - \phi e_1\right\}$ be a global orthonormal basis of smooth eigenvectors for $h$ on $N$ with
$$h R=0\,, \qquad h e_1=-\lambda e_1\,, \qquad h e_2=\lambda e_2\,,$$
where $\lambda\in C^\infty(N)$ is the positive square-root of the function $\la^2$. Let $\eta_1, \eta_2$ be the 1-forms $g$-dual to $e_1$ and $e_2$, respectively, so that $\left\{\alpha, \eta_1, \eta_2\right\}$ is a global $C^\infty$ basis of 1-forms on $N$.

First, using Equation~\eqref{nablaReeb}, we have
\begin{equation}\label{nablaeiR}
\nabla_{e_1} R=-\phi e_1-\phi h e_1=(1-\lambda) e_2\,, \qquad \nabla_{e_2} R = -(1+\lambda) e_1\,.
\end{equation}
Expanding $\nabla_{R} e_i$ in the orthonormal basis $\{R, e_1, e_2\}$, and using $\nabla_{R} R=0$, it follows that
\begin{equation}\label{nablaRei}
\nabla_{R} e_1=a e_2 \quad \text { and } \quad \nabla_{R} e_2=-a e_1\,,
\end{equation}
where $a:=g\left(\nabla_{R} e_1, e_2\right)$ is a $C^\infty$ function on $N$. Now, by using~\eqref{nablaRei} and the definition of the orthonormal frame, we obtain
$$
\left(\nabla_{R} h\right) e_1=\nabla_{R} h e_1-h\left(\nabla_{R} e_1\right)= -R(\lambda) e_1 - 2 a \lambda e_2\,,
$$
and analogously
$$\left(\nabla_{R} h\right) e_2=R(\lambda) e_2-2 a \lambda e_1\,.$$
Using the fact that $R$ is geodesic and $hR=0$, we also have $\left(\nabla_{R} h\right) R=\nabla_R h R - h\nabla_R R=0$, so we finally obtain
$$
\nabla_{R} h = 2 a h \phi+ R(\ln \lambda) h\,,
$$
thus implying, in view of Equation~\eqref{critmetric} and Lemma \ref{1stintegral}, that $a=1$ because $g$ is a critical metric by assumption.

Next, using~\eqref{nablaRei} with $a=1$ and the identities~\eqref{nablaeiR}, a straightforward computation allows us to obtain
$$
\begin{aligned}
& \eta_1 \wedge d \eta_1\left(R, e_1, e_2\right)=-d \eta_1\left(R, e_2\right)= g\left(e_1, \nabla_{R} e_2-\nabla_{e_2} R\right) = \lambda >0\,, \\
& \eta_2 \wedge d \eta_2\left(R, e_1, e_2\right)=d \eta_2\left(R, e_1\right)=- g\left(e_2, \nabla_{R} e_1-\nabla_{e_1} R\right) = - \lambda <0\,, \\
& \eta_1 \wedge d \eta_2\left(R, e_1, e_2\right)=-d \eta_2\left(R, e_2\right)= g\left(e_2, \nabla_{R} e_2-\nabla_{e_2} R\right)=0\,, \\
& \eta_2 \wedge d \eta_1\left(R, e_1, e_2\right) = d \eta_1\left(R, e_1\right) = - g\left(e_1, \nabla_{R} e_1-\nabla_{e_1} R\right) = 0\,.
\end{aligned}
$$
Therefore, the 1-forms $\left(\eta_1, \eta_2\right)$ define a bi-contact structure on $N$ satisfying \eqref{bic0} (this is a special $(-1)$-Cartan structure in the sense of~\cite{perro}), and the Reeb field $R$ is at the intersection of the contact planes $\ker \eta_1$ and $\ker \eta_2$.  Moreover, by combining~\eqref{nablaeiR} and~\eqref{nablaRei} we get $[R, e_1]=\lambda e_2$ and $[R, e_2]=\lambda e_1$, so the property~\eqref{bic1} follows by evaluating it on the orthonormal basis $\{R, e_1, \e_2\}$.

Finally, since $\left\{\alpha, \eta_1, \eta_2\right\}$ is a global orthonormal co-frame, it follows that the metric $g$ has the expression~\eqref{g_eta}, which completes the proof of the lemma.

In the case of non-trivial eigenspace line bundles $\ell_1$ and $\ell_2$,
we take pairs of local framings  $(e_1, e_2)$ and  $-(e_1, e_2)= (-e_1, -e_2)$.
Then, according to a local choice of $\pm(e_1, e_2)$,  $\pm(\eta_1, \eta_2)$
is chosen. It is then straightforward to see that all the computations above remain the same if $(e_1, e_2)$ is replaced by $-(e_1, e_2)$.
\end{proof}

Let $(M,\alpha)$ be a closed contact $3$-manifold and assume that $g$ is a critical compatible metric. Consider the function $\Lambda:=\lambda^2 \in C^\infty(M)$ as defined before. If $\Lambda\equiv0$ then Equation~\eqref{eq.lam} implies that $\mathcal{L}_R g=0$, $R$ is then a Killing field of $g$, so by definition $g$ is a Sasakian metric. Accordingly, the theorem is proved in this case and from now on we can assume that $\Lambda$ is not identically zero.

If the zero set of $\Lambda$,
\[
Z_\Lambda:=\{p\in M:\Lambda(p)=0\}
\]
is empty (a case which is considered by Perrone in~\cite[Theorem 3.1]{perr}), then, by Lemma~\ref{bi-contactsup}, $R$ is supported by a $C^\infty$ bi-contact structure, and therefore by~\cite{mitsu}, it is a conformally (or projectively) Anosov field. Since any volume-preserving conformally Anosov flow is in fact Anosov (see e.g.~\cite[Corollary 1.3.]{hozo1}), and $R$ preserves the volume form $\alpha\wedge d\alpha$, we conclude that $R$ is Anosov. Recalling that the function $\lambda^2$ is a first integral of $R$ by Lemma~\ref{1stintegral}, and no Anosov flow can exhibit non-constant first integrals~\cite{Agui}, we also infer that the function $\lambda^2$ is constant everywhere; in particular, this ensures that the supporting bi-contact structure $(\zeta_1, \zeta_2)$ of $R$ is a calibrated one.

Let us now prove that $R$ is $C^\infty$-conjugate to an algebraic Anosov flow. Assume that we are in the case where the framing $(e_1, e_2)$ is global. Defining the smooth vector fields 
 $e_s:=\tfrac{1}{\sqrt{2}}(e_1+e_2)$ and $e_u:=\tfrac{1}{\sqrt{2}}(e_1-e_2)$, we deduce from Equation~\eqref{bic1} that:
$$
[R, e_s] = \lambda e_s\,, \qquad [R, e_u] = -\lambda e_u\,,
$$
so $e_s$ and $e_u$ span the strong stable and unstable Anosov directions; recall that $e_1=\eta_1^{\sharp_g}$ and $e_2=\eta_2^{\sharp_g}$ and $\la$ is constant. In particular, the strong stable and unstable foliations which are generated by the vector fields $e_s$ and $e_u$, respectively, are smooth.
Since $d\alpha(e_u, e_s)=2$, we infer that
$$[e_s, e_u]=2R + f_s e_s + f_u e_u\,,$$
for some smooth functions $f_s$ and $f_u$ on $M$.
Taking the time $t$ flow $\phi_t=\mathrm{exp} (tR)$ of the Reeb vector field $R$,
these two functions satisfy (as in~\cite[Theorem A]{Green})
$$
[e_s, e_u] = [e^{-t\lambda}e_s, e^{t\lambda}e_u] = {\phi_t}_*[e_s, e_u]
= 2R + e^{-t\lambda}f_s\circ\phi_{-t} \, e_s +
e^{t\lambda}f_u\circ\phi_{-t} \, e_u\,,
$$
and thus we have  $f_s=e^{-t\lambda}f_s\circ\phi_{-t}$
and $f_u=e^{t\lambda}f_u\circ\phi_{-t}$ for all $t\in\RR$.
This immediately implies $f_s=f_u\equiv0$, and therefore we obtain the relations
$$
[R, e_s] = \lambda e_s\,, \quad [R, e_u] = -\lambda e_u\,,
\quad [e_s, e_u] = 2R\,,
$$
with $\lambda$ a positive constant, which define a $C^\infty$ basis for the Lie algebra $\mathfrak{sl}(2;\RR)$
with a hyperbolic element $R$. This allows us to conclude that $R$ is not only orbitally conjugate to an algebraic model, but it is actually $C^\infty$-conjugate.  

Assume now that we do not have a global $(e_1, e_2)$. Given a pair of local framings $(e_1, e_2)$ and  $-(e_1, e_2)$ we
define the positive/negative pair of local framings of vector fields
$(e_s, e_u)$ and  $(-e_s, -e_u)$ by
$e_s:=\tfrac{1}{\sqrt{2}}(e_1+e_2)$ and $e_u:=\tfrac{1}{\sqrt{2}}(e_1-e_2)$
and by
$-e_s:=\tfrac{1}{\sqrt{2}}(-e_1-e_2)$ and $-e_u:=\tfrac{1}{\sqrt{2}}(-e_1+e_2)$.
Then, setting the functions $f_s$ and $f_u$ for $(-e_s, -e_u)$ to be
 $-f_s$ and $-f_u$, the argument remains the same.
 
This completes the proof in the case $Z_\Lambda = \emptyset$.

The last case to treat is $Z_\Lambda\neq \emptyset$. We claim that this case cannot occur. Indeed, consider the open set $ M\setminus Z_\Lambda$, and take a connected component $U$ (open by definition). The function
$$\psi:=\Lambda|_U$$
is not a constant (otherwise it would be 0) and it is positive; obviously $\psi\in C^\infty(U)$. By Sard's theorem, $\psi$ has a full measure set of regular values so let us take a regular value $c>0$. By the preimage theorem,
$$\Sigma:=\psi^{-1}(c)$$
is a surface, which is compact because it cannot approach the boundary of $U$ (otherwise $c=0$). We can safely assume that $\Sigma$ is connected; if not, just take one of its components. Since the (non-vanishing) Reeb field $R$ is tangent to the level sets of $\psi$ by Lemma~\ref{1stintegral}, $\Sigma$ is diffeomorphic to a torus or a Klein bottle (because the Euler characteristic is $\chi(\Sigma)=0$ by the Poincar\'e-Hopf theorem). In addition, $\Sigma$ is cooriented by the gradient $\nabla \psi \neq 0$, so finally we deduce that $\Sigma \cong \Tt^2$.

Since $\nabla\psi \neq 0$ on $\Sigma$, and $\Sigma$ is compact, $\nabla\psi \neq 0$ in a neighbourhood of $\Sigma$. The implicit function theorem then easily implies that neighbouring level sets are also diffeomorphic to $\Tt^2$, so there is a compact set $N \subset U$ that is diffeomorphic to $\Tt^2 \times [c-\delta,c+\delta]$ and is fibred by the level sets of the function $\psi$.

Now we apply Lemma~\ref{bi-contactsup} to the compact contact $3$-manifold $N$ and we deduce that the Reeb field $R|_{N}$ is supported by a bi-contact structure. We can then use the following lemma, which is a local version of a conformally Anosov characterization of vector fields with supporting bi-contact structures~\cite{mitsu}. We conclude that $R|_N$ is a conformally Anosov flow that does not preserve a volume, which contradicts the fact that $R|_N$ preserves the volume form $\alpha\wedge d\alpha|_N$. This contradiction means that the set $Z_{\Lambda}$ must be empty and hence $\Lambda$ is a positive
constant on $M$, thus completing the proof of the first part of the main theorem.

\begin{lemma}\label{localm}
Let $N$ be a compact $3$-manifold with smooth boundary.
If we have a bi-contact structure on $N$,
and the vector field at the intersection of the two contact bundles
is tangent to $\partial N$,  then it is conformally Anosov.
However, the flow is not Anosov and, in particular, it does not preserve a volume.
\end{lemma}
\begin{proof} We just recall the arguments
in~\cite{mitsu} to prove the first part.
Since the vector field is tangent to the boundary,
it defines a complete flow and we can take the time-$1$ map $\Phi:=\varphi_{1}$ of the flow $\varphi_t$ generated by the field, which is a global $C^\infty$ diffeomorphism of $N$.
The bi-contact structure gives
a global frame $\{e_1,e_2\}$
on the normal bundle $\nu\varphi =TN/T\varphi$ to the flow.
Fix such a frame and consider the bundle automorphism $\tilde\Phi$
induced from the linear map $D\Phi$.
We can adjust $\tilde\Phi$ by pointwise positive scalar multiplication
in such a way that,
with respect to the frame $\{e_{1},e_2\}$,
$$\tilde\Phi_{x}:\nu\varphi_{x}\to\nu\varphi_{\varphi_{1}(x)}$$
is represented by a matrix of $\mathit{SL}(2;\RR)$ for each $x\in N$.
The compactness of $N$ implies that there exists a positive constant
$c>0$ such that any of the entries of the matrix representing
$\tilde\Phi_{x}$ is greater than $c$ for any $x\in N$.
Then the lemma for the trace of the iterated dynamics in $\mathit{SL}(2;\RR)$ proved in~\cite[Section~4]{mitsu} applies and we infer that
$$\mathit{tr}\tilde\Phi^{k}_{x}\to \infty$$
when $k \to \infty$. This shows that by the action of the flow, the contact structures are flown
to converge to an invariant plane field, thus proving that the intersection field is conformally Anosov.

Finally, if the flow were Anosov, then the only invariant
surfaces would be the weak stable or weak unstable leaves of the Anosov foliations, which cannot be
the boundary tori, so we conclude that the flow cannot be Anosov, nor it cannot preserve a volume (because any conformally Anosov flow that is volume-preserving is necessarily Anosov). This completes the proof of the lemma.
\end{proof}

\subsection{The proof of the sufficiency part}\label{SS.conv}
Let $(M, \alpha)$ be a closed contact 3-manifold. Let us show that if $R$ is supported by a $C^\infty$ calibrated bi-contact structure, then there exists a critical ($C^\infty$) metric compatible with $\alpha$.

Assume that $R$ is supported by the calibrated bi-contact structure $(\zeta_1, \zeta_2)$ with constant
$\lambda$ and global defining 1-forms $(\eta_1, \eta_2)$, so that we have $\alpha \wedge \eta_1 \wedge \eta_2 = \Omega$ and
\begin{equation}\label{4omega}
\begin{split}
&\eta_1 \wedge d \eta_1 = \lambda\Omega\,, \qquad \eta_2 \wedge d \eta_2 = - \lambda\Omega\,, \\
& \eta_1 \wedge d \eta_2= 0\,, \qquad \eta_2 \wedge d \eta_1 = 0\,.
\end{split}
\end{equation}
Consider the $C^\infty$ Riemannian metric
$$g:= \alpha\otimes\alpha + \eta_1\otimes\eta_1 + \eta_2\otimes\eta_2\,,$$
whose volume element is $\mathrm{vol}_g=\alpha \wedge \eta_1 \wedge \eta_2 = \Omega$, by assumption. Since in addition $\abs{\alpha}_g = 1$ we deduce that
$$\ast_gd \alpha = 2\alpha\,,$$
so $g$ is a metric compatible with the contact form $\alpha$.

We claim that $g$ is a critical compatible metric. To this end, let us evaluate the relations \eqref{4omega} on the (positive) orthonormal frame $\{R, e_1, e_2\}$, $g$-dual to $\{\alpha, \eta_1, \eta_2\}$. We easily obtain
$$\eta_1([R, e_2])=\eta_2([R, e_1])=\lambda\,,$$
and
$$\eta_1([R, e_1])=0\,,\qquad \eta_2([R, e_2])=0\,.$$
Since we also have
$$g([R, e_i], R)=g(\nabla_Re_i,R)-g(\nabla_{e_i}R,R)=0$$
for $i=1,2$, where we have used that $g(\nabla_Re_i,R)=-g(e_i,\nabla_RR)=0$ and $g(\nabla_{e_i}R,R)=\frac12\nabla_{e_i}|R|^2_g=0$, we finally obtain
\begin{equation}\label{structureq}
\begin{split}
&[R, e_1] = \lambda e_2\,, \\
&[R, e_2] = \lambda e_1\,.
\end{split}
\end{equation}

Next, let us define a smooth $(1,1)$-tensor $\phi$ by
$$\phi R=0\,,\qquad \phi e_1 = -e_2\,,\qquad \phi e_2 = e_1\,.$$
It is easy to check that it satisfies Equation~\eqref{assocmet}, so $\phi$ is an endomorphism associated to the compatible metric $g$. Moreover we can compute the $(1,1)$-tensor $h=\frac12\mathcal L_R\phi$ using~\eqref{structureq}, and we obtain
\begin{equation}\label{hconverse}
hR = 0, \qquad he_1=-\lambda e_1\,, \qquad he_2 = \lambda e_2\,,
\end{equation}
from which we immediately deduce $\abs{h}_g^2=2\lambda^2$.
Injecting Equations~\eqref{structureq} and~\eqref{nablaReeb} in the right hand side of
$\nabla_R e_i = [R, e_i] + \nabla_{e_i} R$, and using Equation~\eqref{hconverse}, we finally obtain:
$$
\nabla_R e_1= e_2 , \qquad \nabla_R e_2 = -e_1\,.
$$
Together with~\eqref{hconverse}, this allows us to compute:
$$
(\nabla_R h)e_1= \nabla_R he_1 - h\nabla_R e_1= \nabla_R (-\lambda e_1) - h e_2= -R(\lambda)e_1 -2 \lambda e_2=  -R(\lambda)e_1 +2h\phi e_1,
$$
and analogously: $(\nabla_R h)e_2= R(\lambda)e_2 +2h\phi e_2$. We then infer that Tanno's equation~\eqref{critmetric} is satisfied along $e_1$ and $e_2$ because $R(\lambda)=0$ (recall that $\lambda$ is a constant by the definition of calibrated bi-contact structure) and it is trivially satisfied along $R$. This shows that $g$ defined previously is a critical compatible metric for $\alpha$. 

In the case where the calibrated bi-contact structure is defined by local $\pm (\eta_1, \eta_2)$ the local computations remain the same and the metric $g$ defined above is still global because it is independent of the choice of sign. This completes the proof of the equivalence in the main theorem.

We finally observe that if the Reeb field of $(M,\alpha)$ is $C^\infty$-conjugate to an algebraic Anosov flow modeled on $\widetilde{SL}(2,\mathbb R)$ then it is supported by a calibrated bi-contact structure, and hence there exists a critical compatible metric. This is the content of the following:

\begin{remark}\label{R.alg}
Let us start with the genuine algebraic case of $\widetilde{SL}(2, \RR)$. For an arbitrary constant $\lambda>0$, in the Lie algebra $\mathfrak{sl}(2, \RR)$ seen as the tangent space of $SL(2, \RR)$ at the identity, consider the basis:
$$
R :=\frac{\lambda}{2} \left(
\begin{array}{cc}
 1 & 0 \\
 0 & -1 \\
\end{array}
\right)\,, \quad
e_1 := \sqrt{\frac{\lambda }{2}} \left(
\begin{array}{cc}
 0 & 1 \\
 1 & 0 \\
\end{array}
\right)\,, \quad
e_2 := - \sqrt{\frac{\lambda }{2}} \left(
\begin{array}{cc}
 0 & -1 \\
 1 & 0 \\
\end{array}
\right)\,,
$$
which satisfies the commutation relations
$$
[R, e_1]=\lambda e_2\,, \quad [e_1, e_2]=-2 R\,, \quad [e_2, R]=-\lambda e_1\,.
$$
Promote this basis to a global frame on $SL(2, \RR)$ by left translation and consider a (left invariant) metric $g$ for which this frame is orthonormal.  For the dual co-frame $\{\alpha,\eta_1, \eta_2\}$ the comutation relations yield:
\begin{equation}\label{SL2algebra}
d\alpha= 2 \eta_1 \wedge \eta_2\,, \quad
d\eta_1= -\lambda \alpha \wedge \eta_2\,, \quad
d\eta_2= -\lambda \alpha \wedge \eta_1\,,
\end{equation}
which implies, in particular, that $\alpha$ is a contact form with associated Reeb field $R$ supported by the calibrated bi-contact structure $(\eta_1, \eta_2)$.  In view of the above proof, $g$ is a critical compatible metric, cf. also  \cite[Proposition 4.2]{perr}. Since this construction is left invariant, it is inherited by any compact quotient of $\widetilde{SL}(2, \RR)$. Finally, if a Reeb-Anosov flow is $C^\infty$-conjugate to this algebraic model up to covering, then the construction of a supporting calibrated bi-contact structure is immediate (by pulling-back the contact forms $(\eta_1,\eta_2)$ using the diffeomorphism of the conjugation), and a critical compatible metric then exists.
\end{remark}

\subsection{Global minimizers}\label{SS.min}

In this final section we prove the last statement of Theorem \ref{T.main}, which strengthens the stability property (local minimality) of critical metrics found by Deng \cite{deng}. In view of the results established previously, the proof is essentially a computation of the Chern-Hamilton energy for the case that the Reeb field is supported by a calibrated bi-contact structure.

\begin{proposition}
Let $g$ be a critical compatible metric on $(M,\alpha)$. Then it is a global minimizer of the Chern-Hamilton energy functional.
\end{proposition}

\begin{proof}
Obviously, the Sasakian metrics are global minima of the Chern-Hamilton functional (because $\mathcal L_Rg=0$). Then let us focus on the other class of critical compatible metrics stated in Theorem~\ref{T.main}, which corresponds to the case $\lambda=c$ for some positive constant $c$. We follow the notation introduced in Section~\ref{SS.dir} without further mention. Let $g$ be a critical compatible metric on $(M,\alpha)$ and $(\zeta_1,\zeta_2)$ the (calibrated) bi-contact structure obtained in Lemma~\ref{bi-contactsup}. Assume that the defining contact 1-forms $(\eta_1, \eta_2)$ are global with dual frame $(e_1,e_2)$. The other case will follow easily as in the previous subsection. On the contact distribution $\ker \alpha$ with the frame given by $e_s:=\frac{1}{\sqrt 2}(e_1 + e_2)$ and $e_u:=\frac{1}{\sqrt 2}(e_1 - e_2)$, we consider the dual co-frame $\{\eta_u,\eta_s\}$, so that the expression~\eqref{g_eta} of the critical compatible metric $g$, cf. Lemma~\ref{bi-contactsup}, translates into
$$g =\alpha \otimes \alpha +
\eta_u \otimes \eta_u + \eta_s \otimes \eta_s \, .$$
It is easy to check that the value of the Chern-Hamilton energy for this metric is
$$E(g)=8\lambda^2 \Vol(M)\,,$$
where $\lambda=c>0$ is a constant.

Now, we notice that a general Riemannian metric compatible with the contact form $\alpha$ has the following expression:
$$
\widetilde{g} = \alpha \otimes \alpha + p\eta_u \otimes \eta_u  + r(\eta_u \otimes \eta_s + \eta_s \otimes \eta_u) + q\eta_s \otimes \eta_s\,,
$$
where $p$, $q$ and $r$ are $C^\infty$ functions on $M$ satisfying $p>0$, $q>0$ and $pq-r^2=1$. This last condition ensures that $\mathrm{vol}_{\tilde g} =\mathrm{vol}_{g}$ (as both metrics are compatible with $\alpha$). It then follows from Lemma~\ref{bi-contactsup} and the definition of $\eta_u$ and $\eta_s$ that $\cL_R \alpha = 0$, $\cL_R \eta_u = \lambda \eta_u$ and $\cL_R \eta_s = -\lambda \eta_s$. Accordingly, in terms of the co-frame $\{\alpha,\eta_u,\eta_s\}$, we easily obtain the following expressions
$$
\cL_R \widetilde{g} =
\begin{pmatrix}
0 & 0 & 0 \\
0 & R(p)+2\lambda p & R(r) \\
0 & R(r) & R(q) - 2\lambda q \\
\end{pmatrix} \,\,\,\,\,\,
{\mathrm{and}} \,\,\,\,\,\,
\widetilde{g} =
\begin{pmatrix}
1 & 0 & 0 \\
0 & p & r \\
0 & r & q \\
\end{pmatrix}\,,
$$
where, as usual, $R(f)$ denotes the action for the Reeb vector field $R$ on the function $f\in C^\infty(M)$.
With these expressions the energy density is given as
\[
|\cL_R \widetilde{g}|_{\widetilde g}^2
=\tr(\cL_R \widetilde{g} \cdot \widetilde{g}^{-1} \cdot \cL_R \widetilde{g}
\cdot  \widetilde{g}^{-1})
=:8\tilde\lambda^2
\,.
\]
Then by a direct 
computation, we obtain
$$
\tilde\lambda^2 = \left(r^2+1\right)\lambda ^2 +\left[r R(r)-  p R(q)\right]\lambda +
\tfrac{1}{4} \left[R(r)^2-R(p) R(q)\right]\,.
$$
Substituting $q$ and $R(q)$ from the identity $pq-r^2=1$, we get
$$
\tilde\lambda^2 = \left(r^2+1\right)\lambda ^2 +
\left((r^2+1) R(\ln p)- rR(r)\right)\lambda +
\tfrac{1}{4}\left(r R(\ln p)-R(r)\right)^2+\tfrac{1}{4} R(\ln p)^2\,.
$$
Finally, since $\int_M R(\ln p) \mathrm{vol}_g = 0$ (due to the fact that $\Div_g R = 0$), we infer
\begin{equation*}
\begin{split}
\tfrac{1}{8}(E(\widetilde{g}) - E(g)) = & \int_M ({\tilde \lambda}^2 -\lambda^2) \mathrm{vol}_g \\
= & \int_M \! \big\{r^2\lambda ^2 +
r\left(r R(\ln p)- R(r)\right)\lambda +
\tfrac{1}{4}\left(r R(\ln p)-R(r)\right)^2+\tfrac{1}{4} R(\ln p)^2 \big\} \mathrm{vol}_g\\
= & \int_M \! \big\{ \left[r\lambda + \tfrac{1}{2}\left(r R(\ln p)-R(r)\right)\right]^2+\tfrac{1}{4} (R(\ln p))^2 \big\} \mathrm{vol}_g  \geqslant 0\,,
\end{split}
\end{equation*}
thus showing that indeed the Chern-Hamilton energy has a global minimum at $g$. In fact, the equality in this expression is achieved if and only if $r=0$, $p=\kappa>0$ a positive constant, and $q=1/\kappa$. This follows from the fact that $R$ is Anosov (and hence it does not admit any non-trivial first integrals), and that the equation $2\lambda r=R(r)$ only admits the solution $r=0$ because
\[
\int_M r^2 \mathrm{vol}_g = \frac{1}{4\lambda}\int_M R(r^2)\mathrm{vol}_g=0\,,
\]
where we have used that $\Div_g R = 0$. This completes the proof of the proposition.
\end{proof}

\begin{remark}[Some curvature properties of critical metrics] It is not hard to check that the sectional curvatures along the planes $span\{R, e_s\}$ and $span\{R, e_u\}$ are both equal to $1 - \lambda^2$. In particular, if $\lambda=1$ they are both zero.
This should be compared with the Chern-Hamilton original requirement for critical metrics: ``the sectional curvature of all planes at a given point perpendicular to the contact bundle $\ker\alpha$ are equal'' (cf. their conjecture in~\cite{ch}). Additionally, in the Anosov case, $span\{R, e_s\}$ and $span\{R, e_u\}$ are each tangent to a foliation, and the leaves of this foliation are minimal, i.e., their mean curvature is $0$.
\end{remark}


\section{Applications}\label{S.appli}

In this final section we present some easy applications of Theorem~\ref{T.main}.
 A first straightforward consequence of our main theorem is:
\begin{corollary}\label{coroNONO}
Let $M$ be a closed $3$-manifold which is neither Sasakian nor it does support algebraic Anosov flows that are not suspensions. Then no contact structure on $M$ admits a critical compatible metric. The topological classification of these manifolds appears in Remark~\ref{detailsMain}.
\end{corollary}

Since the 3-torus $\mathbb{T}^3$ cannot be endowed with a Sasakian metric~\cite{geigeSasa, itoh} and it is not diffeomorphic to a compact quotient of $\widetilde{SL}(2,\mathbb R)$, this corollary implies that the 3-torus is a concrete counterexample to the generalized Chern-Hamilton conjecture~\ref{conjectgeneral}.

\begin{example}
On $\mathbb T^3$ we have a family of contact structures given by the standard contact forms:
\begin{equation}\label{standardT3}
\eta_m = \sin(mx_3)\dif x_1 + \cos(mx_3)\dif x_2\,, \qquad m\in \ZZ\,,
\end{equation}
which correspond to the integer part of the curl spectrum with respect to the flat metric (i.e. $\ast \dif \eta_m = m \eta_m$). The rescaled form $\alpha = \frac{m}{2}\eta_m$ admits the flat metric $g_0=\frac{m^2}{4}(dx_1^2 + dx_2^2 + dx_3^2)$ as a compatible metric (with $\theta= 2$). Although this metric is somehow natural for the contact forms $\eta_m$, it turns out that it is not critical.
\end{example}

Another application of our main result shows that no ``chaotic'' Reeb flow (in the sense of positive topological entropy) on $\mathbb S^3$ admits a critical compatible metric:

\begin{corollary}
Let $R$ be the Reeb field of some contact form on $\mathbb S^3$. Then if $R$ has positive topological entropy, the corresponding contact form does not admit a critical compatible metric.
\end{corollary}
\begin{proof}
$R$ cannot be algebraic Anosov because $\mathbb S^3$ is not diffeomorphic to a compact quotient of $\widetilde{SL}(2,\mathbb R)$. If $R$ is a Reeb field of a Sasakian structure on $\mathbb S^3$, it is a Killing field of the corresponding compatible (Sasakian) metric. Since isometries (for some Riemannian metric) always have zero topological entropy~\cite[Exercise 2.5.4]{brin}, we finally conclude that $R$ must have zero topological entropy as well. The corollary follows from the dichotomy in Theorem~\ref{T.main}.
\end{proof}

Our final application shows a surprising connection between a global topological property of the contact structure (being overtwisted) and the existence of critical compatible metrics:

\begin{corollary}
Assume that $(M,\zeta)$ is a closed contact $3$-manifold that is overtwisted. Then it does not admit a critical compatible metric.
\end{corollary}
\begin{proof}
Assume that $(M,\alpha)$ admits a critical compatible metric $g$ for some defining contact form $\alpha$. It cannot be Sasakian because Sasakian structures are always tight~\cite{bel}, so according to Theorem~\ref{T.main}, the Reeb flow associated to $\alpha$ must be Anosov with positive constant $\lambda$. It then follows from~\cite[Corollary~6.2]{perro} that $(M,\zeta)$ is universally tight~\cite{hozo}, thus yielding a contradiction.
\end{proof}

\section*{Acknowledgements}

\noindent The authors are grateful to S.~Hozoori for clarifying to us the current status of the generalized Chern-Hamilton conjecture and for drawing our attention to Ref.~\cite{Green}. When preparing a final version of this article, he informed us in a private communication that he had obtained independently an equivalent result to our main theorem, which will be the content of a future article. We also thank to M.R.R.~Alves for hints regarding the topological entropy. This work has received funding from Grant-in-Aid JP21H00985 and Grant-in-Aid JP21K18579 (Y.M.) funded by JSPS KAKENHI, and the grants CEX2023-001347-S, RED2022-134301-T and PID2022-136795NB-I00 (D.P.-S.) funded by MCIN/AEI, and Ayudas Fundaci\'on BBVA a Proyectos de Investigaci\'on Cient\'ifica 2021 (D.P.-S.).

\bibliographystyle{amsplain}

\end{document}